\UseRawInputEncoding
\documentclass[10pt]{article}
\usepackage{relsize}

\usepackage[dvips]{color}
\usepackage{epsfig}
\usepackage{float,amsthm,amssymb,amsfonts}
\usepackage{ amssymb,amsmath,graphicx, amsfonts, latexsym}
\usepackage{xcolor}
\begin{document}
\theoremstyle{plain}
\newtheorem{theorem}{{\bf Theorem}}[section]
\newtheorem{corollary}[theorem]{Corollary}
\newtheorem{lemma}[theorem]{Lemma}
\newtheorem{proposition}[theorem]{Proposition}
\newtheorem{remark}[theorem]{Remark}

\theoremstyle{definition}
\newtheorem{defn}{Definition}
\newtheorem{definition}[theorem]{Definition}
\newtheorem{example}[theorem]{Example}
\newtheorem{conjecture}[theorem]{Conjecture}

\def\im{\mathop{\rm Im}\nolimits}
\def\dom{\mathop{\rm Dom}\nolimits}
\def\rank{\mathop{\rm rank}\nolimits}
\def\nullset{\mbox{\O}}
\def\ker{\mathop{\rm ker}\nolimits}
\def\implies{\; \Longrightarrow \;}

\def\GR{{\cal R}}
\def\GL{{\cal L}}
\def\GH{{\cal H}}
\def\GD{{\cal D}}
\def\GJ{{\cal J}}

\def\set#1{\{ #1\} }
\def\z{\set{0}}
\def\Sing{{\rm Sing}_n}
\def\nullset{\mbox{\O}}

\title{On the algebraic structure of the Schr\"{o}der monoid}
\author{\bf  Muhammad Mansur Zubairu\footnote{Corresponding Author. ~~Email: \emph{mmzubairu.mth@buk.edu.ng}}, Abdullahi Umar and  Fatma Salim Al-Kharousi   \\
\it\small  Department of Mathematics, Bayero  University Kano, P. M. B. 3011, Kano, Nigeria\\
\it\small  \texttt{mmzubairu.mth@buk.edu.ng}\\[3mm]
\it\small Department of Mathematical Sciences,\\
\it\small Khalifa University, P. O. Box 127788, Sas al Nakhl, Abu Dhabi, UAE\\
\it\small  \texttt{abdullahi.umar@ku.ac.ae}\\[3mm]
\it\small  Department of Mathematics,\\
\it\small College of Science,\\
\it\small Sultan Qaboos University.\\
\it\small \texttt{fatma9@squ.edu.om}}
\date{\today}
\maketitle\

\begin{abstract}
 Let $[n]$ be a finite chain $\{1, 2, \ldots, n\}$, and let $\mathcal{LS}_{n}$ be the semigroup consisting of all isotone and order-decreasing partial transformations on $[n]$. Moreover, let $\mathcal{SS}_{n} = \{\alpha \in \mathcal{LS}_{n} : \, 1 \in \textnormal{Dom } \alpha\}$ be the subsemigroup of $\mathcal{LS}_{n}$, consisting of all transformations in $\mathcal{LS}_{n}$ each of whose domain contains $1$. For $1 \leq p \leq n$, let
$K(n,p) = \{\alpha \in \mathcal{LS}_{n} : \, |\im \, \alpha| \leq  p\}$
and
$M(n,p) = \{\alpha \in \mathcal{SS}_{n} : \, |\im \, \alpha| \leq p\}$
be the two-sided ideals of $\mathcal{LS}_{n}$ and $\mathcal{SS}_{n}$, respectively. Furthermore, let
${RLS}_{n}(p)$ and ${RSS}_{n}(p)$ denote the Rees quotients of $K(n,p)$ and $M(n,p)$, respectively. It is shown in this article that for any
$S \in \{\mathcal{SS}_{n}, \mathcal{LS}_{n}, {RLS}_{n}(p), {RSS}_{n}(p)\}$,
$S$ is abundant and idempotent generated for all values of $n$. Moreover, the ranks  of the Rees quotients
${RLS}_{n}(p)$ and ${RSS}_{n}(p)$ are shown to be equal to the ranks  of the two-sided ideals
$K(n,p)$ and $M(n,p)$, respectively. Finally, these ranks are computed to be
$\sum\limits_{k=p}^{n} \binom{n}{k} \binom{k-1}{p-1}$ and
$\binom{n-1}{p-1}2^{n-p}$, respectively.
 \end{abstract}


\section{Introduction and Preliminaries}
  For a natural number $n$, denote $[n]$ to be the finite chain $\{1,2, \ldots ,n\}$. A map $\alpha$ with its domain and range being subsets of $[n]$ (or with the domain being the entire set $[n]$  and the range being a subset of $[n]$)  is referred to as a \emph{partial} \emph{transformation} (resp.,  \emph{full transformation}). The notations $\mathcal{P}_{n}$ and $\mathcal{T}_{n}$ usually represent \emph{the semigroups of all partial and  full transformations}, respectively. A transformation $\alpha\in \mathcal{P}_{n}$ is said to be an \emph{ isotone} map (resp., an \emph{anti-tone} map) if  (for all $x,y \in \dom\,\alpha$) $x\leq y$ implies $x\alpha\leq y\alpha$ (resp., $x\alpha\geq y\alpha$); \emph{order decreasing} if (for all $x\in \dom\,\alpha$) $x\alpha\leq x$. The notations  $\mathcal{DP}_n$ and $\mathcal{OP}_n$  shall denote  \emph{the semigroup of order-decreasing partial transformations} on $[n]$ and  \emph{the semigroup of all isotone partial transformations} on $[n]$, respectively. As in \cite{auc}, we shall refer to $\mathcal{PC}_{n}$ (\emph{semigroup of all isotone order-decreasing partial transformation} on $[n]$) as the \emph{large} \emph{Schr\"{o}der} monoid and we shall denote it as: \begin{equation}\label{qn111}\mathcal{LS}_{n}= \mathcal{OP}_n\cap \mathcal{DP}_n  .\end{equation} \noindent These monoids  have been extensively studied in various contexts, see for example \cite{zua,  gu1, gm, al1, al2, al3, al4, al5}. The composition of two elements $\alpha $ and $\gamma$ in $\mathcal{P}_{n}$ is defined as $x(\alpha\circ\gamma)=((x)\alpha)\gamma$ for all $x\in\dom\, \alpha$. Without  ambiguity, we shall be using the notation $\alpha\gamma$ to denote $\alpha\circ\gamma$. We shall also use the notations $1_{[n]}$, $\im \alpha$, $\dom \alpha$, $h(\alpha)=|\im \, \alpha|$ to denote the identity map on $[n]$, the image set of a map $\alpha$, the domain set of the map $\alpha$ and the height of $\alpha$, respectively. Furthermore, let
$P$ denote a linearly  ordered partition of $[n]$ in the sense that, for any two sets  $A$ and $B$ in $P$, we write $A<B$ if each element in $A$ is less than every element in $B$.

 Now let
\begin{equation}\label{qn1}
 	\mathcal{SS}_{n} = \{\alpha \in \mathcal{LS}_{n} : 1 \in \textnormal{Dom } \alpha \}
 \end{equation}
\noindent be the set of all maps in $\mathcal{LS}_{n}$ each of whose domain contains  1 and
\begin{equation}\label{qn2}
	\mathcal{SS}^{\prime}_n = \{\alpha \in \mathcal{LS}_{n} : 1 \notin \text{Dom } \alpha\}
\end{equation}
 \noindent be the set of all maps in $\mathcal{LS}_{n}$ each of whose domain do not contains  1. In other words, $\mathcal{SS}^{\prime}_n$ is the set complement of $\mathcal{SS}_{n}$.

 The monoid  $\mathcal{LS}_{n}$ first appeared in Ganyushkin and  Mazorchuk \cite {gmv}, where it was shown that it is idempotent-generated. Moreover, the combinatorial properties of the semigroup  have been explored in \cite{al3}, where  it was shown that the size (or order) of $\mathcal{LS}_{n}$ corresponds to the \emph{large} (or \emph{double}) \emph{Schr\"{o}der number}: \[s_{0}=1, \quad s_{n}= \frac{1}{n+1} \sum\limits_{r=0}^{n}\binom{n+1}{n-r}\binom{n+r}{r} \quad (n\geq 1).\] The set $\mathcal{SS}_{n}$ and its complement $\mathcal{SS}_{n}^{\prime}$ were initially introduced by Laradji and Umar \cite{al5}, who showed that both are subsemigroups of $\mathcal{LS}_{n}$. Interestingly, these two semigroups were found to have the same size, which coincides with the (\emph{small}) \emph{Schr\"{o}der number}: \[s_{n}= \frac{1}{2(n+1)} \sum\limits_{r=0}^{n}\binom{n+1}{n-r}\binom{n+r}{r}.\]  As in \cite{al5}, we shall refer to the semigroup $\mathcal{SS}_{n}$, as the \emph{small} \emph{Schr\"{o}der} monoid.

  Moreover, for $1\le p\le n$,  let \begin{equation} \label{kn} K(n,p)=\{\alpha\in  \mathcal{LS}_{n}: \, |\im \, \alpha|\le p\}\end{equation}
  \noindent and
  \begin{equation}\label{mn} M(n,p)=\{\alpha\in  \mathcal{SS}_{n}: \, |\im \, \alpha|\le p\}\end{equation}
  \noindent be the two sided ideals of $\mathcal{LS}_{n}$  and $\mathcal{SS}_{n}$, respectively, consisting of all decreasing isotone maps with a height of no more than $p$.

  Furthermore, for $p\geq 1$,  let \begin{equation}\label{knn} {RLS}_{n}(p)= K(n,p)/ K(n, p-1)  \end{equation}
  \noindent  be the Rees quotient semigroup of $K(n,p)$, and for $p\geq 2$

  \begin{equation}\label{mnn} {RSS}_{n}(p)= M(n,p)/M(n, p-1)  \end{equation}
\noindent   be the Rees quotient semigroup of $M(n,p)$. The elements of ${RLS}_{n}(p)$ (or ${RSS}_{n}(p)$) can be considered as the elements of $\mathcal{LS}_{n}$ (or $\mathcal{SS}_{n}$) of exactly height $p$. The product of two elements of ${RLS}_{n}(p)$ (or ${RSS}_{n}(p)$) is $0$ if their product in ${RLS}_{n}(p)$ (or ${RSS}_{n}(p)$) has a height strictly less than $p$, otherwise it is in ${RLS}_{n}(p)$ (or ${RSS}_{n}(p)$).

   The algebraic and rank  properties of these subsemigroups have not been studied to our knowledge, see [\cite{al5}, Remark 4.1]. In this paper we are going to study certain algebraic and rank properties of these  semigroups. For more details about basic  terms and concepts in semigroup theory, see the books of  Howie \cite{howi} and Higgins \cite{ph}.

\indent   Following the approach outlined in \cite{HRS}, every $\alpha\in \mathcal{LS}_{n} $  can be represented as
\begin{equation}\label{1}\alpha=\begin{pmatrix}A_1&\ldots&A_p\\a_1&\ldots&a_p\end{pmatrix}   \,  (1\le p\le n),\end{equation}  where  $a_{i}\leq \min A_{i}$ for all $1\leq i\leq p$ and  $A_i$ $(1\le i\le p)$  denote equivalence classes defined by the relation $\textnormal{ker }\alpha=\{(x, y)\in \dom \, \alpha\times \dom \, \alpha: \,  x\alpha=y\alpha\}$, we shall denote this collection by $\textnormal{\bf Ker }\alpha=\{A_1, A_2, \ldots, A_p\}$.  Furthermore, $\textnormal{\bf Ker }\alpha$ is linearly ordered (i.e., for $i<j$, $A_{i}<A_{j}$ if and only if $a<b$ for all $a\in A_{i}$ and $b\in A_{j}$). Moreover, we may without loss of generality assume that $1\leq a_{1}<a_{2}<\ldots<a_{p}\leq n$, since $\alpha$ is an isotone map.

 It is important to mention that the domain of each element in $\mathcal{SS}_{n}$  contains $1$, in particular, $1\in A_{1}$, and so, each element in $\mathcal{SS}_{n}$ of height $1\leq p\leq n$ can be expressed as:
 \begin{equation} \label{eq3}
	\alpha = \begin{pmatrix}A_1&A_2&\ldots& A_p\\1&a_2&\ldots& a_p\end{pmatrix}.
\end{equation}

\section{Regularity, Green's relations and starred Green's relations}

In a semigroup $S$,  an element $a\in S$  is said to be \emph{regular} if there is $b$ in $S$ such that $a=aba$ and $S$ is said to be a \emph{regular semigroup} if every element of $S$ is regular.
When faced with a new type of transformation semigroup, the initial algebraic inquiry typically involves determining the characteristics of its Green's equivalences. These relations are commonly utilized to categorize elements within a semigroup. For definition of these relations, we recommend that the reader consults Howie \cite{howi}. In semigroup theory, there are five Green's relations, namely $\mathcal{L,R,D , J\ \text{and } H}$. It is a known fact in finite semigroups that the relations $\mathcal{D }$ and $\mathcal{J}$ are equivalent (see [\cite{howi}, Proposition 2.1.4]). Therefore, we will focus on characterizing the relations $\mathcal{L,R,D \, \text{and } H}$ on the large and small  Schr\"{o}der monoids $\mathcal{LS}_{n}  \ \text{and }  \mathcal{SS}_{n}$, respectively.

From this point forward in this section, we shall  refer to $\alpha$ and $\beta$ in $\mathcal{LS}_{n}$ as
 \begin{equation} \label{eqq3}
	\alpha = \begin{pmatrix}A_1&\ldots& A_p\\a_{1}&\ldots& a_p\end{pmatrix} \text{and} \  \beta = \begin{pmatrix} B_1 &  \ldots & B_p \\ b_{1} & \ldots  & b_p \end{pmatrix}  \, (1\leq p\leq n)
\end{equation}
\noindent and $\alpha$ and $\beta$ in $\mathcal{SS}_{n}$ as
 \begin{equation} \label{eqq4}
	\alpha = \begin{pmatrix}A_1&A_2&\ldots& A_p\\ 1&a_2&\ldots& a_p\end{pmatrix} \text{and} \  \beta = \begin{pmatrix} B_1 & B_2 & \ldots &  B_p \\ 1 & b_2& \ldots &  b_p \end{pmatrix}  \, (1\leq p\leq n).
\end{equation}

 Now let $S\in \{\mathcal{LS}_{n}, \,  \mathcal{SS}_{n} \}$. Then we have the following theorem.

 \begin{theorem}\label{l}
 Let $S\in \{\mathcal{LS}_{n}, \,  \mathcal{SS}_{n} \}$ and let $\alpha,\beta \in S $  be as in   \eqref{eqq3} or \eqref{eqq4}. Then $\alpha\mathcal{L}\beta$ if and only if
 $\im \, \alpha=\im \, \beta$ \emph{(}i.e., $a_i = b_i$ for $1\leq i\leq p$\emph{)} and $\min A_i =  \min B_i$ for all $1\leq i\leq p$.
 \end{theorem}

 \begin{proof}
	 The proof going forward resembles the proof in  [\cite{umar}, Lemma 2.2.1(2)].
	
 Conversely, suppose that  $\im \, \alpha=\im \, \beta$ and  $\min A_i = \min B_i$ for  all  $1\leq i\leq p$.

Let $t_i = \min A_i$ and  $h_i = \min B_i$ for $1 \le i\le p$. Now if $\alpha, \beta\in \mathcal{LS}_{n}$, then define $\gamma_{1}, \gamma_{2}$ as:

\begin{equation}
	\gamma_1 = \begin{pmatrix}A_1&\ldots& A_p\\t_{1}&\ldots& t_p\end{pmatrix}  \ \text{and } \gamma_{2} = \begin{pmatrix} B_1 &  \ldots &  B_p\\ h_{1} &  \ldots &  h_p \end{pmatrix}.
\end{equation}
\noindent  If $\alpha, \beta\in \mathcal{SS}_{n}$, then we can use the definition of   $\gamma_{1}, \gamma_{2}$ as above after substituting $t_{1}=1=h_{1}$.

In both scenarios, it is evident that  $\gamma_{1}, \gamma_{2} \ \in S$ and $\alpha = \gamma_{1}\beta,\ \beta = \gamma_{2}\alpha$. Thus, ($\alpha$,$\beta$) $\in \mathcal{L}$, as required.
\end{proof}
\begin{theorem}\label{r}
 Let $S\in \{\mathcal{LS}_{n}, \mathcal{SS}_{n} \}$. Then $S$ is $\mathcal{R}-$trivial.
 \end{theorem}
 \begin{proof} $\mathcal{LS}_{n}$ is  known to be $\mathcal{R}$ trivial by [\cite{ph1}, Theorem 4.2] and so $\mathcal{SS}_{n}$ is  $\mathcal{R}-$trivial follows from the fact that $\mathcal{LS}_{n}$ is   $\mathcal{R}$ trivial and   $\mathcal{R}(\mathcal{SS}_{n})\subseteq \mathcal{R}(\mathcal{LS}_{n})\cap (\mathcal{SS}_{n} \times \mathcal{SS}_{n}).$
 \end{proof}

  As a consequence of the above theorem, we readily have the following corollaries.
\begin{corollary}
On the semigroup $S\in \{\mathcal{LS}_{n}, \mathcal{SS}_{n} \}$, $\mathcal{H} = \mathcal{R}$.
\end{corollary}

 \begin{corollary}\label{rem1} Let $\alpha \in S\in \{\mathcal{LS}_{n}, \mathcal{SS}_{n}\}$. Then $\alpha$ is regular if and only if $\alpha$ is an idempotent. Hence, the semigroup $S \in \{\mathcal{LS}_{n}, \mathcal{SS}_{n}\}$ is nonregular.
 \end{corollary}
\begin{proof} The result follows from the fact that in an  $\mathcal{R}$-trivial semigroup, every nonidempotent element is not regular.
 \end{proof}

 \begin{theorem}
On the semigroup $S\in \{\mathcal{LS}_{n}, \mathcal{SS}_{n} \}$, $ \mathcal{D} = \mathcal{L}$.	
	\end{theorem}
\begin{proof}
The result follows from the fact that $S$ is $\mathcal{R}$-trivial from Theorem \ref{r}, and that $\mathcal{D}=\mathcal{L}\circ \mathcal{R}.$
\end{proof}

As a consequence of the three theorems above, we deduce the following characterizations of Green's equivalences on the semigroup $S$ in $\{{RSS}_{n}(p), \, {RLS}_{n}(p), \, M(n,p), \, K(n,p) \}$.

 \begin{theorem} Let $S\in \{{RSS}_{n}(p), \, {RLS}_{n}(p), \, M(n,p), \, K(n,p) \}$ and  let $\alpha, \, \beta \in S$ be as in   \eqref{eqq3} or \eqref{eqq4}.
 Then \begin{itemize} \item[(i)] $\alpha \mathcal{L} \beta$ if and only if $\im \, \alpha = \im \, \beta$ \emph{(}i.e., $a_i = b_i$ for $1 \leq i \leq p$\emph{) }and $\min A_i = \min B_i$ for all $1 \leq i \leq p$; \item[(ii)] $S$ is $\mathcal{R}$-trivial; \item[(iii)] $\mathcal{H} = \mathcal{R}$;  \item[(iv)] $\mathcal{D} = \mathcal{L}$.\end{itemize}  Hence, for $p \geq 3$, the semigroup $S$ is nonregular.
\end{theorem}

 If a semigroup is not regular, it is customary to examine the starred Green's relations in order to classify the algebraic class to which it belongs. Therefore, we will now proceed to characterize the starred analogues of Green's equivalences on these semigroups. For the definitions of these relations we recommend to the reader, Fountain \cite{FOUN2}.

There are five starred Green's equivalences, namely: $\mathcal{L}^*$, $\mathcal{R}^*$, $\mathcal{D}^*$, $\mathcal{J}^*$, and $\mathcal{H}^*$. The relation $\mathcal{D}^*$ is the join of $\mathcal{L}^*$ and $\mathcal{R}^*$, while $\mathcal{H}^*$ is the intersection of $\mathcal{L}^*$ and $\mathcal{R}^*$. A semigroup $S$ is said to be \emph{left abundant} if each $\mathcal{L}^*$-class contains an idempotent; it is said to be \emph{right abundant} if each $\mathcal{R}^*$-class contains an idempotent; and it is said to be \emph{abundant} if each $\mathcal{L}^*$-class and each $\mathcal{R}^*$-class of $S$ contains an idempotent. These classes of semigroups were introduced by Fountain \cite{FOUN, FOUN2}.

Many classes of transformation semigroups have been shown to be either left abundant, right abundant, or abundant; see for example \cite{al1, um,umar,  quasi, ua3, zm1}. Before we characterize the starred Green's relations, we need the following definition and lemmas from \cite{quasi}: A subsemigroup $U$ of $S$ is called an \emph{inverse ideal} of $S$ if for all $u \in U$, there exists $u^{\prime} \in S$ such that $uu^{\prime}u = u$ and both $u^{\prime}u$ and $uu^{\prime}$ are in $U$.

 \begin{lemma}[\cite{quasi}, Lemma 3.1.8.]\label{inv1}  Every inverse ideal $U$ of a semigroup $S$ is abundant.
 \end{lemma}

 \begin{lemma} [\cite{quasi}, Lemma 3.1.9.]  \label{inv2}  Let $U$ be an inverse ideal of a semigroup $S$. Then \begin{itemize} \item[(1)]  $\mathcal{L}^{*} (U) = \mathcal{L}(S) \cap (U \times U)$; \item[(2)] $\mathcal{R}^{*}( U) = \mathcal{R}(S) \cap(U \times U)$; \item[(3)] $\mathcal{H}^{*}( U) = \mathcal{H}(S) \times (U \times U).$\end{itemize}
 \end{lemma}

 We now have the following result.
 \begin{theorem}\label{inv} Let \(\mathcal{LS}_{n}\)  be as defined in \eqref{qn111}. Then \(\mathcal{LS}_{n}\) is an inverse ideal of  $\mathcal{P}_{n}$.
 \end{theorem}
 \begin{proof} Let $\alpha\in \mathcal{LS}_{n}$ be as expressed in \eqref{1}, and let $t_{i}=\min A_{i}$ for all $1\leq i\leq p$. Now define $\alpha^{\prime}$ as: \[\alpha^{\prime}=\begin{pmatrix}
a_1  & \ldots & a_p\\
t_1   & \ldots & t_p
\end{pmatrix} .\]
\noindent Clearly, $\alpha^{\prime}$ is in $\mathcal{P}_{n}$. Notice that:

\begin{align*}\alpha\alpha^{\prime}\alpha &=\begin{pmatrix}
A_1  & \ldots & A_p\\
a_1   & \ldots & a_p
\end{pmatrix}\begin{pmatrix}
a_1  & \ldots & a_p\\
t_1   & \ldots & t_p
\end{pmatrix}\begin{pmatrix}
A_1  & \ldots & A_p\\
a_1   & \ldots & a_p
\end{pmatrix}\\&= \begin{pmatrix}
A_1  & \ldots & A_p\\
a_1   & \ldots & a_p
\end{pmatrix}=\alpha. \end{align*}
\noindent Moreover, \[\alpha^{\prime}\alpha=\begin{pmatrix}
a_1  & \ldots & a_p\\
t_1   & \ldots & t_p
\end{pmatrix}\begin{pmatrix}
A_1  & \ldots & A_p\\
a_1   & \ldots & a_p
\end{pmatrix}=\begin{pmatrix}
a_1  & \ldots & a_p\\
a_1   & \ldots & a_p
\end{pmatrix}=\text{1}_{\im \, \alpha}\in \mathcal{LS}_{n},\]\noindent and also \[\alpha\alpha^{\prime}=\begin{pmatrix}
A_1  & \ldots & A_p\\
a_1   & \ldots & a_p
\end{pmatrix}\begin{pmatrix}
a_1  & \ldots & a_p\\
t_1   & \ldots & t_p
\end{pmatrix}
=\begin{pmatrix}
A_1  & \ldots & A_p\\
t_1   & \ldots & t_p
\end{pmatrix}\in E(\mathcal{LS}_{n})\subset \mathcal{LS}_{n}.\] \noindent Thus,  $\mathcal{LS}_{n}$ is an inverse ideal of $\mathcal{P}_{n}$, as required.
 \end{proof}
 \begin{remark}\label{gg} By letting $a_{1}=t_{1}=1$ in the above theorem and its proof, we deduce that  $\mathcal{SS}_{n}$ is an inverse ideal of $\mathcal{P}_{n}$.
 \end{remark}

 Consequently, we have the following result.

\begin{theorem}
	Let $\mathcal{LS}_{n} \ \text{and } \mathcal{SS}_{n}$ be as defined in \eqref{qn111} and \eqref{qn1}, respectively and let $S\in \{ {\mathcal{LS}_{n}}, \mathcal{SS}_{n} \}$. Then $S$ is abundant.
\end{theorem}

\begin{proof}
	 The result follows from Theorem \ref{inv} (resp., Remark \ref{gg}) and Lemma \ref{inv1}.
\end{proof}

 \begin{theorem} \label{a1}
 Let $S\in \{\mathcal{LS}_{n}, \mathcal{SS}_{n} \}$, then  for $\alpha, \beta\in S$ we have:
 \begin{itemize}
   \item[(i)] $\alpha\mathcal{L}^*\beta$  if and only $\im  \alpha = \im  \beta$;
   \item[(ii)] $\alpha\mathcal{R}^*\beta$ if and only if $\ker  \alpha = \ker  \beta$;
   \item[(iii)] $\alpha\mathcal{H}^*\beta$ if and only if $\alpha=\beta$;
   \item[(iv)] $\alpha\mathcal{D}^*\beta$ if and only if $|\im  \alpha| = |\im   \beta|$.
 \end{itemize}
 \end{theorem}

\begin{proof}
\begin{itemize} \item[(i)] and (ii) follow from Theorem \ref{inv}, Lemma \ref{inv2} and [\cite{howi}, Exercise 2.6.17], while (iii) follows from (i) and (ii) and the fact that $\alpha$ and $\beta$ are isotone.
\item[(iv)] Let's assume that $\alpha\mathcal{D}^{*}\beta$. Thus by (\cite{howi}, Proposition 1.5.11), there exist elements $\gamma_{1},~\gamma_{2}, \ldots,~\gamma_{2n-1}\in ~S$ such that $\alpha\mathcal{L}^{*}\gamma_{1}$, $\gamma_{1}\mathcal{R}^{*}\gamma_{2}$, $\gamma_{2}\mathcal{L}^{*}\gamma_{3},\ldots,$ $\gamma_{2n-1}\mathcal{R}^{*}\beta$ for some $n\in ~ \mathbb{{N}}$. Consequently, from  (i) and (ii), we deduce that $\im~\alpha=\im~\gamma_{1}$, ${\ker}~\gamma_{1}={\ker}~\gamma_{2}$, $\im~\gamma_{2}=\im~\gamma_{3},\ldots,$ $\ker~\gamma_{2n-1}=\ker~\beta$. Now it follows that $|\im~\alpha|=|\im~\gamma_{1}|=|\dom~\gamma_{1}/ \ker~\gamma_{1}|=|\dom~\gamma_{2}/ \ker~\gamma_{2}|=\ldots=|\dom~\gamma_{2n-1}/ \ker~\gamma_{2n-1}|=|\dom~\beta/ \ker~\beta|=|\im~\beta|.$

Conversely, suppose that $|\im~\alpha|=|\im~\beta|$ where \begin{equation*}\label{2} \alpha=\left(\begin{array}{ccc}
                                                                            A_{1}  & \ldots & A_{p} \\
                                                                            a_{1} & \ldots & a_{p}
                                                                          \end{array}
\right)\text{ and } \beta=\left(\begin{array}{ccc}
                                                                            B_{1}  & \ldots & B_{p} \\
                                                                            b_{1} & \ldots & b_{p}
                                                                          \end{array}
\right).\end{equation*}

Now define \begin{equation*}\label{2} \delta=\left(\begin{array}{ccc}
                                                                            A_{1}  & \ldots & A_{p} \\
                                                                            {1} & \ldots & {p}
                                                                          \end{array}
\right)\text{ and } \gamma=\left(\begin{array}{ccc}
                                                                            B_{1}  & \ldots & B_{p} \\
                                                                            {1} & \ldots & {p}
                                                                          \end{array}
\right).\end{equation*}

\noindent Clearly, $\delta$ and $\gamma$ are in $S$. Notice that $\ker \, \alpha= \ker \, \delta$, $\im \, \delta=\im \, \gamma$ and $\ker \, \gamma=\ker \, \beta$. Thus by (i) and (ii) we see that $\alpha \mathcal{R}^{*} \delta \mathcal{L}^{*} \gamma \mathcal{R}^{*} \beta$.

 \noindent Similarly,  define $\delta=\left(\begin{array}{ccc}
                                                                            n-p+{1}  & \ldots & n \\
                                                                            a_{1} & \ldots & a_{p}
                                                                          \end{array}
\right)$ and  $\gamma=\left(\begin{array}{ccc}
                                                                            n-p+1  & \ldots & n \\
                                                                            b_{1} & \ldots & b_{p}
                                                                          \end{array}
\right)$. Clearly, $\delta$ and $\gamma\in S$. Moreover,  notice that $\im \, \alpha=\im  \, \delta$,   $\ker \, \delta= \ker \, \gamma$,  $\im \, \gamma=\im \,  \beta$. Thus by (i) and (ii) we have $\alpha \mathcal{L}^{*} \delta \mathcal{R}^{*} \gamma \mathcal{L}^{*}\beta$. Hence, by (\cite{howi}, Proposition 1.5.11) it follows that $\alpha\mathcal{D}^{*}\beta$.  The proof is now complete.
\end{itemize}
\end{proof}

\begin{lemma}\label{uaaaa} On the Schr\"{o}der monoids  $\mathcal{LS}_{n}$ and  $\mathcal{SS}_{n}$ \emph{(}$n\geq 3$\emph{)}, we have $\mathcal{D}^{*}=\mathcal{R}^{*}\circ\mathcal{L}^{*}\circ\mathcal{R}^{*}=\mathcal{L}^{*}\circ\mathcal{R}^{*}\circ\mathcal{L}^{*}$.
\end{lemma}
\begin{proof} The sufficiency  follows from the converse of the  proof of (iv) in the above theorem, while for the necessity, we have to prove that   $\mathcal{L}^{*}\circ\mathcal{R}^{*}\neq \mathcal{R}^{*}\circ\mathcal{L}^{*}$. Take \[\alpha=\left(\begin{array}{cc}
                                                                            1  &  2 \\
                                                                            {1} &2
                                                                          \end{array}
\right) \text{ and } \beta=\left(\begin{array}{cc}
                                                                            1  &  3 \\
                                                                            {1} &3
                                                                          \end{array}
\right). \]

\noindent Now define $\delta=\left(\begin{array}{cc}
                                                                            1  &  3 \\
                                                                            {1} &2
                                                                          \end{array}
\right).$ Then clearly $\im \, \alpha=\im \, \delta$ and $\dom \, \delta=\dom \, \beta$, and so  $\alpha \mathcal{L}^{*} \delta \mathcal{R}^{*}\beta$. i.e., $(\alpha, \beta)\in \mathcal{L}^{*} \circ \mathcal{R}^{*}$.

On the other hand, if we have $(\alpha, \beta)\in  \mathcal{R}^{*} \circ \mathcal{L}^{*}$, then  there must exist $\gamma \in\mathcal{SS}_{n} \subseteq \mathcal{LS}_{n}$ such that $\alpha \mathcal{R}^{*} \gamma \mathcal{L}^{*}\beta$. However, this means that $\dom \, \alpha= \dom \, \gamma=\{1,2\}$ and $\im \, \gamma=\im \, \beta=\{1,3\}$, which is impossible. The result now follows.
\end{proof}

\begin{lemma}\label{uaaa} On the semigroups  ${RLS}_{n}(p)$ and  ${RSS}_{n}(p)$, we have $\mathcal{D}^{*}=\mathcal{R}^{*}\circ\mathcal{L}^{*}\circ\mathcal{R}^{*}=\mathcal{L}^{*}\circ\mathcal{R}^{*}\circ\mathcal{L}^{*}$.
\end{lemma}
\begin{proof} The proof is the same as  the proof of  the above lemma.
\end{proof}

As in \cite{FOUN2}, to define the relation $\mathcal{J}^{*}$ on a semigroup $S$, we first denote the $\mathcal{L}^{*}$-class
containing the element $a\in S$  by $L^{*}_{a}$. (The corresponding notation
can be used for the classes of the other relations.) A \emph{left} (resp., \emph{right}) $*$-\emph{ideal} of a
semigroup $S$ is defined to be a \emph{left} (resp., \emph{right}) ideal $I$ of $S$ such that $L^{*}_{a} \subseteq I$ (resp., $R^{*}_{a} \subseteq I$), for all $a \in  I$. A subset $I$ of $S$ is a $*$-ideal of $S$ if it is both  left  and
right $*$-ideal. The \emph{principal $*$-ideal} $J^{*}(a)$ generated by the element $a\in S$  is defined to be the intersection of all $*$-ideals of $S$ to which $a$ belongs. The relation $\mathcal{J}^{*}$ is  defined by the rule that $a \mathcal{J}^{*}  b$ if and only if $J^{*}(a) = J^{*}(b)$, where $J^{*}(a)$ is the principal $*$-ideal generated by $a$.

The next lemma is crucial to  our next investigation about the properties of $\mathcal{J}^{*}$ in the semigroup  $S\in\{\mathcal{LS}_{n}, \mathcal{SS}_{n} \}$.

\begin{lemma}[\cite{FOUN2}, Lemma 1.7]\label{jj}  Let $a$  be an element of a semigroup $S$. Then $b \in J^{*}(a)$ if and only if there are elements $a_{0},a_{1},\ldots, a_{n}\in  S$, $x_{1},\ldots,x_{n}, y_{1}, \ldots,y_{n} \in S^{1}$ such that $a = a_{0}$, $b = a_{n}$, and $(a_{i}, x_{i}a_{i-1}y_{i}) \in \mathcal{D}^{*}$ for $i = 1,\ldots,n.$
\end{lemma}

As in \cite{ua}, we now have the following:

\begin{lemma}\label{jjj} For $\alpha, \, \beta\in S\in\{\mathcal{LS}_{n}, \mathcal{SS}_{n} \}$, let $ \alpha\in J^{*}(\beta)$. Then $\mid \im \, \alpha \mid\leq \mid \im \,\beta \mid$.
\end{lemma}
\begin{proof} Let $ \alpha \in J^{*}(\beta)$. Then, by Lemma \ref{jj}, there exist $\beta_{0}, \beta_{1},\ldots, \beta_{n}$, $\gamma_{1}, \ldots, \gamma_{n}$, $\tau_{1}, \ldots, \tau_{n}$ in $S\in\{\mathcal{LS}_{n}, \mathcal{SS}_{n} \}$ such that $\beta=\beta_{0}$,  $\alpha=\beta_{n}$, and $(\beta_{i}, \gamma_{i}\beta_{i-1}\tau_{i})\in \mathcal{D}^{*}$ for $i =1,\ldots,n.$ Thus, by Lemma \ref{uaaaa}, this implies that
\[\mid\im \,\beta_{i} \mid= \mid\im \, \gamma_{i}\beta_{i-1}\tau_{i} \mid\leq \mid\im \, \beta_{i-1} \mid ,\] \noindent so that
\[\mid \im \, \alpha \mid\leq \mid \im \,\beta \mid,\] \noindent as required.
\end{proof}

\begin{lemma}\label{uaaaaa} On the large and small Schr\"{o}der monoids  $\mathcal{LS}_{n}$ and  $\mathcal{SS}_{n}$, we have $\mathcal{J}^{*}=\mathcal{D}^{*}$.
\end{lemma}
\begin{proof} Notice we need to only show that $\mathcal{J}^{*} \subseteq \mathcal{D}^{*}$ (since $\mathcal{D}^{*} \subseteq \mathcal{J}^{*}$). So, suppose that $(\alpha,\beta) \in \mathcal{J}^{*}$, then $J^{*}(\alpha)=J^{*}(\beta)$, so that $\alpha\in J^{*}(\beta)$ and $\beta\in J^{*}(\alpha)$. However, by Lemma \ref{jjj}, this implies that \[\mid \im \, \alpha \mid \leq \mid \im \, \beta \mid \text{ and } \mid \im \, \beta \mid \leq \mid \im \, \alpha \mid,\] \noindent so that $\mid \im \, \alpha \mid= \mid \im \, \beta \mid$. Thus by Lemma \ref{uaaaa}, we have \[\mathcal{J}^{*} \subseteq \mathcal{D}^{*},\]\noindent as required.

\end{proof}

\begin{lemma}\label{un} On the semigroup $S$ in $\{\mathcal{LS}_{n}, \, \mathcal{SS}_{n}, \, {RSS}_{n}(p), \, {RLS}_{n}(p),  \, M(n,p), \,  K(n,p) \}$, every $\mathcal{R}^{*}-$class  contains a unique idempotent.
\end{lemma}
\begin{proof} This follows from the fact that  \textbf{Ker }$\alpha$ can only admit one image subset of $[n]$ so that $\alpha$ is an idempotent by the decreasing property of $\alpha$.

\end{proof}

\begin{remark}\begin{itemize}
             \item[(i)]  It is now clear that, for each $1\le p \le n$, the number of $\mathcal{R}^{*}-$classes in $J^{*}_{p}=\{\alpha\in \mathcal{LS}_{n}: \, |\im \, \alpha|=p\}$ is equal to the number of all possible partial ordered partitions of $[n]$ into $p$ parts. This  is equivalent to the number of  $\mathcal{R}-$classes  in  $ \{\alpha\in \mathcal{OP}_n: \, |\im \, \alpha|=p\}$, which is known to be $\sum\limits_{r=p}^{n}{\binom{n}{r}}{\binom{r-1}{p-1}}$ from \emph{ [\cite{al3}, Lemma 4.1]}.
             \item[(ii)] If $S\in \{{RSS}_{n}(p), \, {RLS}_{n}(p),  \, M(n,p), \,  K(n,p) \}$. Then the characterizations of the starred Green's relations in Theorem \ref{a1}, also hold in $S$.
           \end{itemize}
\end{remark}

Thus, the semigroup $K(n,p)$, like $\mathcal{LS}_{n}$ is the union of $\mathcal{J}^{*}$ classes \[ J_{o}^{*}, \, J_{1}^{*}, \, \ldots, \, J_{p}^{*}\]
where \[J_{p}^{*}=\{\alpha\in K(n,p): \, |\im \, \alpha|=p\}.\]

Furthermore, $K(n,p)$ has $\sum\limits_{r=p}^{n}{\binom{n}{r}}{\binom{r-1}{p-1}}$ $\mathcal{R}^{*}-$classes and $\binom{n}{p}$ $\mathcal{L}^{*}-$classes in each $J^{*}_{p}$. Consequently, the Rees quotient semigroup ${RLS}_{n}(p)$ has $\sum\limits_{r=p}^{n}{\binom{n}{r}}{\binom{r-1}{p-1}}+1$ $\mathcal{R}^{*}-$classes and $\binom{n}{p}+1$ $\mathcal{L}^{*}-$classes. (The term 1 is derived from the singleton class containing the zero element in every instance.)

 Now, let $J^{*}_{p}=\{\alpha\in \mathcal{SS}_{n}: \, h(\alpha)=p\}$. We compute the number of $\mathcal{R}^{*}$ classes in $J^{*}_{p}$ and the number of idempotents in $\mathcal{SS}_{n}$ in the lemmas below.

 \begin{lemma}  For $1\leq p\leq n$, the number of $\mathcal{R}^{*}-$classes in $J^{*}_{p}$ is \[\sum\limits_{r=p}^{n}{\binom{n-1}{r-1}}{\binom{r-1}{p-1}}.\]
 \end{lemma}
 \begin{proof} Let $\alpha\in \mathcal{SS}_{n}$ be such that $h(\alpha)=p$ and $|\dom \, \alpha|=r$ for $p\leq r\leq n$. Next observe that since $1\in \dom \, \alpha$, then we can choose the remaining $r-1$ elements of $\dom \, \alpha$ from $[n]\setminus \{1\}$ in $\binom{n-1}{r-1}$ ways. Moreover, we can partition $\dom \, \alpha$ into $p$ convex (modulo $\dom \, \alpha$) subsets in $\binom{r-1}{p-1}$ ways. The result follows after multiplying these two binomial coefficients and taking the sum from $r=p$ to $r=n$.
 \end{proof}
  \begin{lemma}\label{ssch} For  $1\le p \le n$, we have $\sum\limits_{r=p}^{n}{\binom{n-1}{r-1}}{\binom{r-1}{p-1}}=\binom{n-1}{p-1}2^{n-p}$.
  \end{lemma}
  \begin{proof} \begin{align*} \sum\limits_{r=p}^{n}{\binom{n-1}{r-1}}{\binom{r-1}{p-1}}=& \sum\limits_{r=p}^{n}{\frac{(n-1)!}{(n-r)!(r-1)!}\cdot\frac{(r-1)!}{(r-p)!(p-1)!}}\\&= \sum\limits_{r=p}^{n}{\frac{(n-1)!}{(n-r)!(r-p)!(p-1)!}}\\&= \sum\limits_{r=p}^{n}{\frac{(n-1)!(n-p)!}{(n-r)!(p-1)!(r-p)!(n-p)!}} \, \, \left(\textnormal{multiplying by   $\frac{(n-p)!}{(n-p)!}$}\right)\\&=\sum\limits_{r=p}^{n}{\frac{(n-1)!}{(p-1)!(n-p)!}\cdot\frac{(n-p)!}{(n-r)!(r-p)!}} \textnormal{ (by spliting and rearranging the fractions)}\\&
  = \sum\limits_{r=p}^{n}{\binom{n-1}{p-1}\binom{n-p}{n-r}}\\&
  = \binom{n-1}{p-1}\sum\limits_{r=p}^{n}{\binom{n-p}{n-r}}\\&= \binom{n-1}{p-1}2^{n-p},
  \end{align*}
  as required.
  \end{proof}
  Now we have the theorem below.
  \begin{theorem} Let $\mathcal{SS}_{n}$ be as defined in \eqref{qn1}. Then  $|E(\mathcal{SS}_{n})|=3^{n-1}$.
  \end{theorem}
  \begin{proof}The result follows from Lemma \ref{ssch} by  summing up $\binom{n-1}{p-1}2^{n-p}$ from $p=1$ to $p=n$.
  \end{proof}

  \begin{remark}
  Notice that $1\in \dom \, \alpha$ for every  $\alpha \in M(n,p)$. Thus $M(n,p)$ has $\binom{n-1}{p-1}2^{n-p}$ $\mathcal{R}^{*}-$classes and $\binom{n}{p}$ $\mathcal{L}^{*}-$classes in each of its $J^{*}_{p}$. Similarly, the Rees quotient semigroup ${RSS}_{n}(p)$ has $\binom{n-1}{p-1}2^{n-p}+1$ $\mathcal{R}^{*}-$classes and $\binom{n}{p}+1$ $\mathcal{L}^{*}-$classes.
\end{remark}

\section{Rank properties}

Let $S$ be a semigroup and $A$ be a nonempty subset of $S$. The \emph{smallest subsemigroup} of $S$ that contains $A$ is called the \emph{  subsemigroup generated by $A$} usually denoted by $\langle A \rangle$. If there exists a finite subset $A$ of a semigroup $S$ such that $\langle A \rangle$ equals $S$, then $S$ is referred to as a \emph{finitely-generated semigroup}. The \emph{rank} of a finitely generated semigroup $S$ is defined as the minimum cardinality of a subset $A$ such that $\langle A \rangle$ equals $S$. i.e.,
\[
\text{rank}(S) = \min\{|A| : \langle A \rangle = S\}.
\]
\noindent If the set $A$  consists exclusively of the idempotents in $S$, then $S$ is called \emph{idempotent-generated} (equivalently, a \emph{semiband}), and the idempotent-rank is  denoted by  $\text{idrank}(S)$.
The monoid  $\mathcal{LS}_{n}$ first appeared in Ganyushkin and  Mazorchuk \cite {gmv}, where it was shown that it is idempotent generated. Moreover, the combinatorial properties of the semigroup  have been explored in \cite{al3}, where  it was shown that the size (or order) of $\mathcal{LS}_{n}$ corresponds to the \emph{large} (or \emph{double}) \emph{Schr\"{o}der number}: \[s_{0}=1, \quad s_{n}= \frac{1}{n+1} \sum\limits_{r=0}^{n}\binom{n+1}{n-r}\binom{n+r}{r} \quad (n\geq 1).\]
\noindent It is important to note that Dimitrova and Koppitz  \cite{dm} examines the rank of the semigroup of all order-preserving and \emph{extensive} (which means order-increasing) partial transformations on a finite chain, denoted by $\mathcal{POE}_{n}$. This monoid can easily be shown to be isomorphic to the large \emph{Schr\"{o}der} monoid $\mathcal{LS}_{n}$ (see \cite{umar}); thus, the rank of $\mathcal{LS}_{n}$ can easily be obtained from [\cite{dm}, Proposition 4.0], by isomorphism. However, we present the result and proof for the sake of completeness and the generating elements. Moreover, Ping \emph{et al.} \cite{png} generalized the results of Dimitrova and Koppitz \cite{dm}, where they obtained the rank of the two-sided ideals of $\mathcal{POE}_{n}$. The rank of the two-sided ideals of $\mathcal{LS}_{n}$ can also be obtained through isomorphism from [\cite{png}, Proposition 2.6]. Furthermore, both articles fail to recognize that each of the objects considered has a minimum generating set (not minimal) since each of the object is an $\mathcal{R}$-trivial semigroup; thus, most of the proofs given in the two articles are belabored. In this section, we provide among other results the proof of the rank properties of these objects.

For a more detailed discussion about  ranks in  semigroup theory, we refer the reader to \cite{hrb, hrb2}. Several authors have explored the ranks, idempotent ranks, and nilpotent ranks of various classes of semigroups of  transformations. Notably, the works of Gomes and Howie \cite{gm, gm2, gm3}, Howie and McFadden \cite{hf}, Garba \cite{g1, g2, gu1},  Umar \cite{umar,  ua1, ua} and Zubairu \emph{et. al.,} \cite{zm1} are here emphasized. The large Schröder monoid $\mathcal{LS}_{n}$ has been shown to be idempotent-generated in [\cite{gmv}, Theorem 14.4.5], where it first appeared. Our aim is to compute the rank of the two sided ideal  $M(n,p)$ of the Schr\"{o}der monoid  $\mathcal{SS}_{n}$, thereby obtaining the rank of  $\mathcal{SS}_{n}$, as special cases. We first note the following definitions and a well known result about decreasing maps from \cite{umar, ua3}.

Let $f(\alpha)$ be the cardinal of
\[F(\alpha) =\{x\in \dom \, \alpha: \, x\alpha=x\} ;\]
 the set of fixed points of the map $\alpha$. Then we have the following lemma.

\begin{lemma}\label{hq} For all order decreasing partial maps $\alpha$ and $\beta$ on $A\subseteq [n]$, $F(\alpha\beta)=F(\alpha)\cap F(\beta)=F(\beta\alpha)$.
\end{lemma}
\begin{proof} If $\alpha$ or $\beta$ is zero (i.e., the empty map), the result follows. Now suppose $\alpha$ and $\beta$ are nonzero partial order decreasing maps. The proof is the same as the proof of Lemma 2. 1.  in \cite{ua3}.
\end{proof}

However, we initiate our examination with the following result about generating elements of  $\mathcal{LS}_{n}$.

\begin{lemma}\label{hq} The large Schr\"{o}der monoid $\mathcal{LS}_{n}$ is idempotent-generated.
\end{lemma}
\begin{proof} Let $\alpha\in \mathcal{LS}_{n}$ be as expressed in  \eqref{1}. If $p=0$, then $\alpha$ is the empty map which is an idempotent, and the result follows.

Now suppose $1\le p\le n$ and let $t_{i}=\min A_{i}$ for $1\le i \le p$. Notice that the maps defined as \[\epsilon=\begin{pmatrix}A_1&\cdots&A_p\\t_1&\cdots&t_p\end{pmatrix}\] \noindent and
\[\epsilon_{i}=\begin{pmatrix}a_1&\cdots & a_{i-1}&\{a_{i},t_{i}\}&t_{i+1}&\cdots &t_p\\a_1&\cdots&a_{i-1}&a_{i}&t_{i+1}&\cdots&t_p\end{pmatrix} \, \, (1\le i\le p)\]
\noindent are idempotents in $J^{*}_{p}$.

Moreover,
\begin{align*}\epsilon\epsilon_{1}\cdots\epsilon_{p}=&
\begin{pmatrix}
A_1 &  \cdots &  A_p \\
t_1 &  \cdots &  t_p
\end{pmatrix}
\begin{pmatrix}
\{t_{1},a_1\} & t_2 & \cdots &  t_p \\
a_1 & t_2 & \cdots &  t_p
\end{pmatrix}
\begin{pmatrix}
a_1 & \{a_{2},t_2\} & t_{3} & \cdots &  t_p \\
a_1 & a_2 & t_{3} & \cdots &  t_p
\end{pmatrix}
\cdots
\begin{pmatrix}
a_1 & \cdots & a_{p-1} & \{a_{p}, t_p\} \\
a_1 &  \cdots & a_{p-1} & a_p
\end{pmatrix}
\\=& \begin{pmatrix}
A_1 & \cdots &  A_p \\
a_1 &  \cdots &  a_p
\end{pmatrix}
= \alpha.
\end{align*}
 The result now follows.
\end{proof}

Notice that in the proof of the above lemma, $|\im \, \alpha|=h(\alpha)=h(\epsilon)=h(\epsilon_{i})=p$ for all $1\le i\le p$. Consequently, we have the following result.

\begin{lemma}\label{hh}
Every element  in $S\in\{{RSS}_{n}(p), \, {RLS}_{n}(p), \, K(n,p), \,  M(n,p) \}$ of height $p$ can be expressed as a product of idempotents in $S$, each of height $p$.
\end{lemma}

The next result shows that the set of nonzero idempotents in ${RLS}_{n}(p)$ (${RSS}_{n}(p)$) is the minimum generating set of ${RLS}_{n}(p)\setminus \{0\}$ (${RSS}_{n}(p)\setminus \{0\}$).
 \begin{lemma} Let $\alpha$, $\beta$ be elements in ${RLS}_{n}(p)\setminus \{0\}$ \emph{(}resp., $\alpha$, $\beta$ in ${RSS}_{n}(p)\setminus \{0\}$\emph{)}. Then $\alpha\beta\in E({RLS}_{n}(p)\setminus \{0\})$ \emph{(}resp., $\alpha\beta\in E({RSS}_{n}(p)\setminus \{0\})$\emph{)} if and only if $\alpha, \, \beta\in E({RLS}_{n}(p)\setminus \{0\})$  and $\alpha\beta=\alpha$ \emph{(}resp., $\alpha, \, \beta\in E({RSS}_{n}(p)\setminus \{0\})$ and $\alpha\beta=\alpha$\emph{)}.
\end{lemma}

\begin{proof} Suppose
$\alpha\beta\in E({RLS}_{n}(p)\setminus \{0\})$ (resp., $\alpha\beta\in E({RSS}_{n}(p)\setminus \{0\})$). Then \[
p= f(\alpha\beta)\leq f(\alpha)\leq |\im \, \alpha|=p,\]
\[
p= f(\alpha\beta)\leq f(\beta)\leq |\im \, \beta|=p.\]

This ensures that \[F(\alpha)=F(\alpha\beta)=F(\beta), \] \noindent and so $\alpha, \, \beta\in E({RLS}_{n}(p)\setminus \{0\})$  and $\alpha\beta=\alpha$ (resp., $\alpha, \, \beta\in E({RSS}_{n}(p)\setminus \{0\})$ and $\alpha\beta=\alpha$).

The converse is obvious.
\end{proof}

It is evident that necessity ensues, as the result of multiplying a non-idempotent element with any other element does not yield a non-zero idempotent,  by Lemma \ref{hq}. Consequently, the rank and idempotent rank of ${RLS}_{n}(p)$ (resp., ${RSS}_{n}(p)$) are equivalent. Therefore, we have now established the key results of this section.

\begin{theorem}\label{pb} Let ${RLS}_{n}(p)$ be as defined in \eqref{knn}. Then  \[\text{rank } {RLS}_{n}(p)= \text{idrank } {RLS}_{n}(p)=|E({RLS}_{n}(p)\setminus\{0\})|=\sum\limits_{r=p}^{n}{\binom{n}{r}}{\binom{r-1}{p-1}}.\]
\end{theorem}
\begin{proof} It follows from the fact that there are $\sum\limits_{r=p}^{n}{\binom{n}{r}}{\binom{r-1}{p-1}}$ $\mathcal{R}^{*}-$classes in ${RLS}_{n}(p)$ and each $\mathcal{R}^{*}-$class contains a unique idempotent, by Lemma \ref{un}.
\end{proof}

\begin{theorem}\label{mnnn} Let ${RSS}_{n}(p)$ be as defined in \eqref{mnn}. Then  \[\text{rank } {RSS}_{n}(p)= \text{idrank } {RSS}_{n}(p)=|E({RSS}_{n}(p)\setminus\{0\})|=\binom{n-1}{p-1}2^{n-p}.\]
\end{theorem}
\begin{proof} It follows from the fact that there are $\binom{n-1}{p-1}2^{n-p}$ $\mathcal{R}^{*}-$classes in ${RSS}_{n}(p)$ and each $\mathcal{R}^{*}-$class contains a unique idempotent, by Lemma \ref{un}.
\end{proof}

The next lemma is crucial to determine the ranks of the Schr\"{o}der monoids $\mathcal{LS}_{n}$  and $\mathcal{SS}_{n}$.  Now  for $1\leq p\leq n$ let \[J^{*}_{p}=\{\alpha\in \mathcal{LS}_{n}: |\im \, \alpha|=p \}.\]

\begin{lemma}\label{lm1} For $0\leq p\leq n-2$, $J^{*}_{p}\subset \langle J^{*}_{p+1}\rangle$. In other words, if $\alpha\in J^{*}_{p}$ then $\alpha\in \langle J^{*}_{p+1}\rangle$ for $1\leq p\leq n-2$.
\end{lemma}
\begin{proof} It suffices to prove that every idempotent of height $p$ can be expressed as product of idempotents of height $p+1$, by Lemma \ref{hq}. Let $\epsilon\in E(J^{*}_{p})$ be expressed as: \[\epsilon= \begin{pmatrix}
A_1 &  \cdots  & A_p \\
t_1 &  \cdots  & t_p
\end{pmatrix},  \] \noindent where $\min A_{i}=t_{i}$ for all $1\le i\le p$.

Now $\epsilon$ is either a full map or a partial map. We shall first consider the case when $\epsilon$ is a full map.

\noindent \textbf{Case 1.} Suppose $\epsilon$ is a full idempotent map of height $p$. Then since $p\le n-2$, then there exists $1\le i\le p$,  such that $|A_{i}|\ge 2$. Thus, there are two subcases to consider.\\

\noindent\textbf{Subcase i.} Suppose $|A_{i}|=2$ and let $A_{i}=\{t_{i}, x_{i_{1}}\}$. Then there exists $1\leq j\leq p$,  such that $|A_{j}|\ge 2$. Now let $A_{j}=\{t_{j}, y_{j_{1}},y_{j_{2}},  \ldots, t_{j+1}-1\}$, where we may suppose without  loss of generality that $i<j$. Then define $\epsilon_{1}$ and $\epsilon_{2}$ as:

\[\epsilon_{1}=\begin{pmatrix}
A_1  & \cdots & A_{i-1} & t_{i} & x_{i_{1}}& A_{i+1}&\cdots & A_{p} \\
t_1  & \cdots & t_{i-1} &t_{i}&x_{i_{1}}& t_{i+1}& \cdots &  t_p
\end{pmatrix}\]
 \noindent and

  \[\epsilon_{2}=\left( \begin{array}{cccccccccccc}
t_{1}  & \cdots & t_{i-1} & \{t_{i}, x_{i_{1}} \} & t_{i+1} & \cdots & t_{j} &  y_{j_{1}}  & t_{j+1} & \cdots & t_{p}\\
t_{1} & \cdots & t_{i-1} & t_{i} & t_{i+1} & \cdots & t_{j} & y_{j_{1}} & t_{j+1} & \cdots &  t_{p}
\end{array}
   \right).\]
 \\
\noindent\textbf{Subcase ii.} Suppose $|A_{i}|>2$, and let $A_{i}=\{t_{i}, x_{i_{1}}, \, x_{i_{2}} \, \ldots, \, t_{i+1}-1\}$. Thus  define $\epsilon_{1}$ and $\epsilon_{2}$ as:

\[\epsilon_{1}=\begin{pmatrix}
A_1  & \cdots & A_{i-1} & t_{i} & \{x_{i_{1}}, \, x_{i_{2}} \, \ldots, \, t_{i+1}-1 \}& A_{i+1}&\cdots & A_{p} \\
t_1  & \cdots & t_{i-1} &t_{i}& x_{i_{1}} & t_{i+1}& \cdots &  t_p
\end{pmatrix}\]
\noindent and

  \[\epsilon_{2}=\left( \begin{array}{cccccccccccc}
t_{1}  & \cdots & t_{i-1} & \{t_{i}, x_{i_{1}} \} & x_{i_{2}} & t_{i+1} & \cdots & t_{p}\\
t_{1}  & \cdots & t_{i-1} & t_{i} &  x_{i_{2}}& t_{i+1} & \cdots &  t_{p}
\end{array}
   \right).\]
   \noindent  Clearly, in either of the subcases, one can easily see that $\epsilon_{1}$ and $\epsilon_{2}$ are in  $E(J^{*}_{p+1})$, and also in each subcase $\epsilon_{1}\epsilon_{2}=\epsilon$.\\

\noindent\textbf{Case 2.} Now suppose $\epsilon$ is strictly partial. Thus $|(\dom \, \epsilon)^{c}|\ge 1$. Let $q\in (\dom \, \epsilon)^{c}$. Then there are three subcases to consider, i.e.,  either (i) there exists $x_{i_{k}}$ and $x_{i_{k+1}}$ in $A_{i}$ (for some $1\le i\le p$ ),  such that $x_{i_{k}}< q< x_{i_{k+1}}$; (ii) or there exists $1\le i\le p-1$,  such that $\max A_{i}<q<\min A_{i+1}$; (iii) or $\max A_{p}< q$. Thus, we consider the three cases separately.

   \textbf{Subcase i.} Suppose there exists $x_{i_{k}}$ and $x_{i_{k+1}}$ in $A_{i}$ (for some $1\le i\le p$ ),  such that $x_{i_{k}}< q< x_{i_{k+1}}$. Let $A_{i}=\{t_{i}, x_{i_{1}}, \ldots, x_{i_{{k}}}, x_{i_{k+1}}, \ldots\}$. Thus define $\epsilon_{1}$, $\epsilon_{2}$ and $\epsilon_{3}$ as:

   \[\epsilon_{1}=\begin{pmatrix}
A_1 &  \cdots & A_{i-1}  & \{ t_{i}, x_{i_{1}}, \, \ldots, \, x_{i_{k}} \}& \{ x_{i_{k+1}}, x_{i_{k+2}}, \, \ldots  \}&A_{i+1}&\cdots & A_{p} \\
t_1 &  \cdots & t_{i-1} &t_{i}& x_{i_{k+1}}& t_{i+1}& \cdots &  t_p
\end{pmatrix};\]

  \[\epsilon_{2}=\left( \begin{array}{ccccccccccc}
t_{1}  & \cdots &  t_{i} & \{q, x_{i_{k+1}} \} & t_{i+1} & \cdots & t_{p}\\
t_{1}  & \cdots &  t_{i} &  q& t_{i+1} & \cdots &  t_{p}
\end{array}
   \right)\]
\noindent and
     \[\epsilon_{3}=\left( \begin{array}{cccccccccccc}
t_{1}  & \cdots & t_{i-1} & \{t_{i}, q \} &  x_{i_{k+1}}  & t_{i+1} & \cdots & t_{p}\\
t_{1}  & \cdots & t_{i-1} & t_{i} &  x_{i_{k+1}}& t_{i+1} & \cdots &  t_{p}
\end{array}
   \right).\]

\textbf{Subcase ii.} Suppose $\max A_{i}< q< \min A_{i+1}$. Now either \textbf{(a)} $|A_{i}|\geq 2$; or \textbf{(b)} $|A_{i}|= 1$ and $|A_{i+1}|\geq 2$; or \textbf{(c)} $|A_{i}|=|A_{i+1}|=1$ and  $|A_{j}|\geq 2$ for some $i\neq j\neq i+1$; or \textbf{(d)} all blocks of $\epsilon$ are singleton, and so there exists $q^{\prime}\in \, (\dom \, \epsilon)^{c}$ such that $\max A_{j}< q< \min A_{j+1}$ for some $1\leq j\leq p-1$.

 \noindent \textbf{(a.)} Suppose $\max A_{i}< q< \min A_{i+1}$ such that $|A_{i}|\geq 2$. Now let $A_{i}=\{t_{i}, x_{i_{1}}, x_{i_{2}}, \dots\}$. Then define  $\epsilon_{1}$, $\epsilon_{2}$ and $\epsilon_{3}$ as:

   \[\epsilon_{1}=\begin{pmatrix}
A_1 &  \cdots & A_{i-1}  & t_{i}& \{ x_{i_{1}}, x_{i_{2}}, \, \ldots,  \}&A_{i+1}&\cdots & A_{p} \\
t_1 &  \cdots & t_{i-1} &t_{i}& x_{i_{1}}& t_{i+1}& \cdots &  t_p
\end{pmatrix};\]

  \[\epsilon_{2}=\left( \begin{array}{ccccccccccc}
t_{1}  & \cdots & t_{i} & \{ x_{i_{1}}, q \} & t_{i+1} & \cdots & t_{p}\\
t_{1}  & \cdots &  t_{i} &  x_{i_{1}}& t_{i+1} & \cdots &  t_{p}
\end{array}
   \right)\]
\noindent and
     \[\epsilon_{3}=\left( \begin{array}{cccccccccccc}
t_{1}  & \cdots & t_{i-1} & \{t_{i}, x_{i_{1}} \} &  q  & t_{i+1} & \cdots & t_{p}\\
t_{1}  & \cdots & t_{i-1} & t_{i} &  q& t_{i+1} & \cdots &  t_{p}
\end{array}
   \right).\]
\noindent  In either of the two subcases above, one can easily verify that $\epsilon_{1}$, $\epsilon_{2}$ and $\epsilon_{3}$ are in $E(J^{*}_{p+1})$, and also \[\epsilon_{1}\epsilon_{2}\epsilon_{3}=\epsilon.\]

\noindent\textbf{(b.)} Suppose $\max A_{i}< q< \min A_{i+1}$ such that, $|A_{i}|=1$ and  $|A_{i+1}|\geq 2$. Now, let $A_{i+1}=\{t_{i+1}, x_{{(i+1)}_{1}}, x_{{(i+1)}_{2}}, \dots\}$. Then define  $\epsilon_{1}$ and $\epsilon_{2}$ as:

 \[\epsilon_{1}=\begin{pmatrix}
A_1 &  \cdots & A_{i}  & t_{i+1}& \{ x_{{(i+1)}_{1}}, x_{{(i+1)}_{2}}, \, \ldots,  \}&A_{i+2}&\cdots & A_{p} \\
t_1 &  \cdots & t_{i} &t_{i+1}& x_{{(i+1)}_{1}}& t_{i+2}& \cdots &  t_p
\end{pmatrix}\]
\noindent and

  \[\epsilon_{2}=\left( \begin{array}{ccccccccccc}
t_{1}  & \cdots & t_{i} &  q  & \{t_{i+1}, x_{{(i+1)}_{1}}\} & t_{i+2}& \cdots & t_{p}\\
t_{1}  & \cdots &  t_{i} &  q & t_{i+1} & t_{i+2}& \cdots &  t_{p}
\end{array}
   \right).\]
   \noindent   Clearly $\epsilon_{1}$ and $\epsilon_{2}$  in $E(J^{*}_{p+1})$ and   $\epsilon_{1}\epsilon_{2}=\epsilon$.

 \noindent \textbf{(c.)} Suppose $\max A_{i}< q< \min A_{i+1}$ such that $|A_{i}|=|A_{i+1}|=1$ and suppose there exists $1\leq j\leq  p-2$ such that $|A_{j}|\geq 2$, where $i\neq j\neq i+1$. Now let $A_{j}=\{t_{j}, y_{j_{1}}, y_{j_{2}}, \, \ldots\}$ and we may suppose  without  loss of generality that $i<j$. Thus $q<y_{j_{1}}$. Now, define $\epsilon_{1}$ and $\epsilon_{2}$ as:
 \[\epsilon_{1}=\begin{pmatrix}
A_1 &  \cdots & A_{i}  &q& A_{i+1}&\cdots & A_{p} \\
t_1 &  \cdots & t_{i} &q&t_{i+1}& \cdots &  t_p
\end{pmatrix}\]
\noindent and

  \[\epsilon_{2}=\left( \begin{array}{cccccccccc}
t_{1}  & \cdots &   t_{j+1}  & y_{j_{1}} & t_{j+2}& \cdots & t_{p}\\
t_{1}  & \cdots  &   t_{j+1} & y_{j_{1}} & t_{j+2}& \cdots &  t_{p}
\end{array}
   \right).\]

\noindent\textbf{(d.)}  Suppose $\max A_{i}< q< \min A_{i+1}$ is such that $|A_{i}|=1$ for all $1\le i \le p$, and suppose there exists $d\in (\dom \, \epsilon)^{c} $,  such that $\max A_{j}< d<\min A_{j+1}$ for some $1< j< p$. We may suppose without loss of generality that $d<q$.

If $i=j$, then  $\max A_{i}< d<q< \min A_{i+1}$. Thus, define  $\epsilon_{1}$ and  $\epsilon_{2}$  as: \[\epsilon_{1}=\begin{pmatrix}
A_1 &  \cdots & A_{i}  &q& A_{i+1}&\cdots & A_{p} \\
t_1 &  \cdots & t_{i} &q&t_{i+1}& \cdots &  t_p
\end{pmatrix}\]
\noindent and

  \[\epsilon_{2}=\left( \begin{array}{cccccccccc}
t_{1}  & \cdots & t_{i} &  d  & t_{i+1}&  \cdots & t_{p}\\
t_{1}  & \cdots & t_{i} &  d & t_{i+1} &  \cdots &  t_{p}
\end{array}
   \right).\]

However, if $i\neq j$ then $j<i$ since $d<q$. So define $\epsilon_{1}$ and  $\epsilon_{2}$  as:
\[\epsilon_{1}=\begin{pmatrix}
A_1 &  \cdots & A_{j}  &d& A_{j+1}&\cdots&A_{i}& \cdots& A_{p} \\
t_1 &  \cdots & t_{j} &d&t_{j+1}& \cdots &t_{i} &\cdots& t_p
\end{pmatrix}\]
\noindent and

  \[\epsilon_{2}=\begin{pmatrix}
t_1 &  \cdots & t_{j}  &t_{j+1}&\cdots&t_{i}& q&t_{i+1}& \cdots& t_{p} \\
t_1 &  \cdots & t_{j} &t_{j+1}& \cdots &t_{i} &q&t_{i+1}&\cdots& t_p
\end{pmatrix}. \]

Thus in all the above cases, it can easily be seen that $\epsilon_{1}$ and  $\epsilon_{2}$ are  in  $E(J^{*}_{p+1})$, and $\epsilon_{1}\epsilon_{2}=\epsilon$.

\noindent\textbf{Case 3.} If $\max A_{p}< q$, then either (i) there exists $1\leq i\leq p$,  such that $|A_{i}|\geq 2$; (ii) or there exists $d\in (\dom \, \epsilon)^{c}$,  such that $q<d$.

\noindent \textbf{(i.)} Suppose there exists $1\leq i\leq p$,  such that $|A_{i}|\geq 2$. Now let $A_{i}=\{t_{i}, x_{i_{1}}, \ldots\}$. Then
define $\epsilon_{1}$ and $\epsilon_{2}$ as:
\[\epsilon_{1}=\begin{pmatrix}
A_1 &  \cdots & A_{i-1}  &t_{i}&\{x_{i_{1}}, x_{i_{2}},  \ldots\}&A_{i+1}& \cdots& A_{p} \\
t_1 &  \cdots & t_{i-1} &t_{i}&x_{i_{1}}& t_{i+1} &\cdots& t_p
\end{pmatrix}\]
\noindent and

  \[\epsilon_{2}=\begin{pmatrix}
t_1 &  \cdots & t_{i-1}  &\{t_{i}, x_{i_{1}}\}&t_{i+1}& \cdots& t_{p}&q \\
t_1 &  \cdots & t_{i-1} &t_{i}& t_{i+1} &\cdots& t_p& q
\end{pmatrix} .\]

\noindent \textbf{(ii.)} Suppose there exists $d\in (\dom \, \epsilon)^{c}$, such that $q<d$. Then define
\[\epsilon_{1}=\begin{pmatrix}
A_1 &  \cdots &  A_{p}&q \\
t_1 &  \cdots & t_p&q
\end{pmatrix}\]
\noindent and

  \[\epsilon_{2}=\begin{pmatrix}
t_1 &  \cdots& t_{p}&d \\
t_1 & \cdots& t_p& d
\end{pmatrix} . \]

\noindent Clearly in (i) and (ii) above, $\epsilon_{1}$ and $\epsilon_{2}$ are clearly  in $E(JJ^{*}_{p+1})$ and also $\epsilon_{1}\epsilon_{2}=\epsilon$.

The proof of the lemma is now complete.
\end{proof}

\begin{remark}\label{nn} The above lemma also holds when the large Schr\"{o}der monoid $\mathcal{LS}_{n}$ is replaced by the small Schr\"{o}der monoid $\mathcal{SS}_{n}$, with a slight modification of the proof by substituting $t_{1} = 1$ in the definition of $\epsilon$.
\end{remark}

Consequently, we have the following (which can equivalently be obtain from [\cite{png}, Proposition 2.6] by isomorphism).

\begin{theorem}\label{knp} Let $K(n,p)$ be as defined in \eqref{kn}. Then  \[\text{rank } K(n,p)= \text{idrank } K(n,p)=|E(J^{*}_{p})|=\sum\limits_{r=p}^{n}{\binom{n}{r}}{\binom{r-1}{p-1}}.\]
\end{theorem}
\begin{proof}Notice that by Lemma \ref{lm1},   $\langle J^{*}_{p} \rangle= K(n,p)$ for all $p$.  Notice also that, $\langle E({RLS}_{n}(p)\setminus\{0\})\rangle= J^{*}_{p}$. The result now follows from Theorem \ref{pb}.
\end{proof}
 We now deduce the following corollary which can also be obtain from [\cite{dm}, Proposition 4] by isomorphism.

\begin{corollary} The  \(\text{rank } \mathcal{LS}_{n}= \text{idrank } \mathcal{LS}_{n}=2n.\)
\end{corollary}
\begin{proof} Notice that by Lemma \ref{lm1}, $\langle J^{*}_{n-1} \rangle=\mathcal{LS}_{n}\setminus J^{*}_{n}=K(n,n-1)$. Notice also that $J^{*}_{n}$ contains only the identity element $1_{[n]}$. Thus  $\text{rank } \mathcal{LS}_{n}= \text{idrank } \mathcal{LS}_{n}= \text{idrank }K(n,n-1)+1$. The result now follows from Theorem \ref{knp}.
\end{proof}

Similarly, we also obtain the next theorem and its corollary:
\begin{theorem}\label{mnp} Let $M(n,p)$ be as defined in \eqref{mn}. Then  \[\text{rank } M(n,p)= \text{idrank } M(n,p)=|E(J^{*}_{p})|=\binom{n-1}{p-1}2^{n-p}.\]
\end{theorem}

\begin{corollary} The \(\text{rank } \mathcal{SS}_{n}= \text{idrank } \mathcal{SS}_{n}=2n-1.\)
\end{corollary}
\begin{proof} The proof is similar to that of the above theorem.
\end{proof}

\noindent{\bf Acknowledgements, Funding and/or Conflicts of interests/Competing interests.} The first named author would like to thank Bayero University and
TETFund (TETF/ES/UNI/KANO/TSAS/2022) for financial support. He would also like to thank Sultan Qaboos University, Oman,  for hospitality during a 1-year postdoc research visit to the institution.

\end{document}